\documentclass[12pt,reqno]{amsart}

\usepackage{amssymb}

\newtheorem{theorem}{Theorem}[section]

\newtheorem{lemma}[theorem]{Lemma}

\numberwithin{equation}{section}

\begin{document}
\baselineskip=15.5pt

\title[Holomorphic Cartan geometries on complex tori]{Holomorphic Cartan geometries on complex 
tori}

\author[I. Biswas]{Indranil Biswas}

\address{School of Mathematics, Tata Institute of Fundamental
Research, Homi Bhabha Road, Mumbai 400005, India}

\email{indranil@math.tifr.res.in}

\author[S. Dumitrescu]{Sorin Dumitrescu}

\address{Universit\'e C\^ote d'Azur, CNRS, LJAD, France}

\email{dumitres@unice.fr}



\date{}

\begin{abstract} 
In \cite{DM} it was asked whether all flat holomorphic Cartan geometries $(G,H)$
on a complex torus are translation invariant. We answer this affirmatively under the
assumption that the complex Lie group $G$ is affine. More precisely,
we show that every holomorphic Cartan geometry of type $(G,H)$, with $G$ a complex 
affine Lie group, on any complex torus is translation invariant.
\end{abstract}

\maketitle

\vspace{0.1cm}

\noindent{{R\'esum\'e.}} Nous d\'emontrons que sur les tores complexes, toutes les 
g\'eom\'etries de Cartan holomorphes model\'ees sur $(G,H)$, avec $G$ groupe de Lie complexe 
affine, sont invariantes par translation.

{\bf Version fran\c caise abr\'eg\'ee.} Dans cette note nous \'etudions les g\'eom\'etries de 
Cartan holomorphes sur les tores complexes. Rappelons que, gr\^ace aux r\'esultats de 
\cite{BM,D1,D2}, les tores complexes sont, \`a rev\^etement fini pr\`es, les seules vari\'et\'es 
de Calabi-Yau qui poss\`edent des g\'eom\'etries de Cartan holomorphes. Il a \'et\'e 
conjectur\'e dans \cite{DM} que sur les tores complexes, toutes les g\'eom\'etries de Cartan 
holomorphes plates sont n\'ecessairement invariantes par translation. Cette conjecture a \'et\'e 
prouv\'ee dans \cite{DM} dans certains cas particuliers (par exemple, pour les tores de 
dimension complexe un et deux et, en toute dimension, dans le cas $G$ nilpotent). Il a 
\'egalement \'et\'e d\'emontr\'e dans \cite{DM} que si on consid\`ere, sur chaque tore complexe, 
l'espace des g\'eom\'etries de Cartan holomorphes plates de mod\`ele $(G,H)$ fix\'e, avec $G$ et 
$H$ groupes alg\'ebriques complexes, les g\'eom\'etries de Cartan invariantes par translation 
forment un sous-ensemble ouvert et ferm\'e (et donc une union de composantes connexes). Dans la 
direction de la conjecture il a aussi \'et\'e prouv\'e dans \cite{Mc} que sur les tores 
complexes, toutes les g\'eom\'etries de Cartan {\it paraboliques} sont invariantes par 
translation.

Dans cet article nous d\'emontrons que {\it sur les tores complexes, toutes les g\'eom\'etries 
de Cartan holomorphes model\'ees sur $(G,H)$, avec $G$ groupe de Lie complexe affine, sont 
invariantes par translation.}

La d\'emonstration de ce th\'eor\`eme repose fortement sur des r\'esultats de \cite{BG}, 
inspir\'e de \cite{Si}. Plus pr\'ecis\'ement, il est d\'emontr\'e dans \cite{BG} que les 
fibr\'es principaux holomorphes $E_G$, de groupe structural complexe affine $G$, au-dessus d'un 
tore complexe $X$, admettent des connexions holomorphes exactement quand ils sont homog\`enes 
(i.e., chaque translation dans $X$ se rel\`eve en un isomorphisme du fibr\'e principal $E_G$). 
Ceci est \'egalement \'equivalent avec la fait que $E_G$ soit pseudostable, avec les deux 
premi\`eres classes de Chern nulles \cite{BG}. De plus, il est montr\'e dans \cite{BG} (en 
adaptant des arguments de \cite{Si}) que ces fibr\'es principaux holomorphes admettent 
\'egalement une connexion canonique plate.

Dans cet article on utilise cette connexion canonique plate sur le fibr\'e principal $E_G$, 
associ\'e \`a la g\'eom\'etrie de Cartan model\'ee sur $(G,H)$, pour relever l'action du tore 
$X$ par translation sur lui-m\^eme en une action qui pr\'eserve la classe d'isomorphisme de 
$E_G$, ainsi que sa connexion canonique plate. On d\'emontre ensuite que cette action pr\'eserve 
\'egalement la connexion holomorphe de $E_G$ et le sous-fibr\'e principal $E_H$ (de groupe 
structural $H$) qui caract\'erisent la g\'eom\'etrie de Cartan. Ceci implique que la 
g\'eom\'etrie de Cartan model\'ee sur $(G,H)$ est invariante par translation.

\section{Introduction}

A classical result proved by Inoue, Kobayashi and Ochiai \cite{IKO}, which was done 
using Yau's proof of Calabi conjecture, shows that a compact complex K\"ahler 
manifold, bearing a holomorphic connection on its holomorphic tangent bundle, admits a 
finite unramified holomorphic covering by some compact complex torus. The pull-back of such a 
holomorphic connection to the covering torus is necessarily translation invariant.

This result was generalized in \cite{BM}, \cite{D1}, \cite{D2} for two different 
classes of holomorphic geometric structures: the rigid geometric structures in 
Gromov's sense \cite{Gr}, and the Cartan geometries. More precisely, any compact 
complex K\"ahler Calabi-Yau manifold bearing a holomorphic rigid geometric structure 
(or a holomorphic Cartan geometry) admits a finite unramified holomorphic covering
by a complex torus.

There are interesting examples of holomorphic rigid geometric structures on complex 
tori that are not translation invariant. They can be constructed using Ghys 
example of holomorphic foliations on complex tori which are not translation 
invariant \cite{Gh}. Let us recall that the main result of \cite{Gh} is a 
classification of codimension one (nonsingular) holomorphic foliations on complex 
tori. The holomorphic foliations 
are defined by the kernel of some global holomorphic 1-form \(\omega\) 
(and hence are translation invariant), except for those complex tori $T$ which admit 
a holomorphic surjective map $\pi$ to an elliptic curve $E$. In the last case one 
can consider a global coordinate $z$ on $E$, a nonconstant meromorphic function 
$u(z)$ on $E$ and the pull-back to $T$ of the meromorphic closed 1-form $u(z)dz$. The
foliation 
given by the kernel of $\Omega\,= \,\pi^*(u(z)dz) + \omega$ extends to all of $T$ as 
a codimension one nonsingular holomorphic foliation; this foliation
coincides with the one given by the
fibration $\pi$ exactly when $\omega$ vanishes on the fibers of $\pi$. This 
foliation is not invariant by all translations in the torus $T$, but only by those 
lying in the kernel of the linear map underlying $\pi$. Consequently, the 
holomorphic rigid geometric structure on $T$ obtained by considering the previous 
holomorphic foliation together with the holomorphic standard flat connection of $T$ 
is not translation invariant (it is invariant only by those translations lying in 
the kernel of the linear map underlying $\pi$).

This note deals with holomorphic Cartan geometries on complex tori. In contrast
with the 
situation of the geometric structures in the previous example, it was conjectured 
in \cite{DM} that {\it all flat holomorphic Cartan geometries on complex tori are 
translation invariant}. The conjecture was proved in \cite{DM} for some particular 
cases --- for example, when the torus is one or two dimensional, or when the
structure group $G$ of the
Cartan the geometry is nilpotent. For $G$ complex algebraic, it was also proved in 
\cite{DM} that, on any torus $T$, translation invariant Cartan geometries form an 
open and closed subset (and, consequently, a union of connected components) in the 
space of Cartan geometries with a given model $(G,H)$ on $T$. Moreover, Theorem 3 
in \cite{Mc} proves that every holomorphic parabolic Cartan geometry on any complex 
torus is translation invariant.

We prove here the following:

{\it Every holomorphic Cartan geometry of type $(G,H)$, with $G$ 
a complex affine group, on any complex torus is translation invariant.}

\section{Preliminaries}

\subsection{Holomorphic Cartan geometries}

Let $G$ be a connected complex Lie group and $H\, \subset\, G$ a complex Lie subgroup.
The Lie algebras of $G$ and $H$ will be denoted by $\mathfrak g$ and $\mathfrak h$
respectively. A holomorphic Cartan geometry of type $(G,\, H)$ on a complex manifold $X$
is a principal $H$--bundle $f\, :\, E_H\, \longrightarrow\, X$ and a $\mathfrak g$
valued holomorphic $1$-form $\omega\ :\, TE_H\, \longrightarrow\, E_H\times {\mathfrak g}$
on the total space of $E_H$, such that
\begin{enumerate}
\item $\omega$ is $H$--equivariant for the adjoint action of $H$ on $\mathfrak g$;

\item $\omega$ is an isomorphism;

\item the restriction of $\omega$ to any fiber of $f$ coincides with the Maurer--Cartan form
associated to the action of $H$ on $E_H$.
\end{enumerate}
Let $\text{At}(E_H)\,=\, TE_H/H\, \longrightarrow\, X$ be the Atiyah bundle for $E_H$.
Let $$E_G\,:=\, E_H\times^H G \,\longrightarrow\, X$$ be the holomorphic principal $G$--bundle
obtained by extending the structure group of $E_H$ using the inclusion of $H$ in $G$. Giving
a form $\omega$ satisfying the above three conditions is equivalent to giving a holomorphic
isomorphism
$$
\beta\, :\, \text{At}(E_H)\, \longrightarrow\, \text{ad}(E_G)\,:=\,E_G\times^G {\mathfrak g}
$$
such that the following diagram is commutative:
\begin{equation}\label{e1}
\begin{matrix}
0 &\longrightarrow & \text{ad}(E_H) &\stackrel{i_H}{\longrightarrow} & \text{At}(E_H) &\longrightarrow &
TX &\longrightarrow & 0\\
&& \Vert &&~ \Big\downarrow\beta && ~\Big\downarrow\sim\\
0 &\longrightarrow & \text{ad}(E_H) &\longrightarrow & \text{ad}(E_G) &\longrightarrow &
\text{ad}(E_G)/\text{ad}(E_H) &\longrightarrow & 0
\end{matrix}
\end{equation}
where the sequence at the top is the Atiyah exact sequence for $E_H$ (see
\cite{At} for the Atiyah exact sequence); the inclusion $\text{ad}(E_H)\,
\hookrightarrow\, \text{ad}(E_G)$ in \eqref{e1} is given by the inclusion of $\mathfrak h$
in $\mathfrak g$. Consider the injective homomorphism
$$\text{ad}(E_H)\, \longrightarrow\, \text{ad}(E_G)\oplus \text{At}(E_H)\, ,\ \
v\, \longmapsto\, (v,\, -i_H(v))\, ,$$
where $i_H$ is the homomorphism in \eqref{e1}.
The corresponding quotient bundle $(\text{ad}(E_G)\oplus \text{At}(E_H))/\text{ad}(E_H)$ is identified
with the Atiyah bundle $\text{At}(E_G)$ for $E_G$. If $\beta$ is a homomorphism as above defining
a holomorphic Cartan geometry on $X$ of type $(G,\, H)$, then the homomorphism
$$
\beta'\, :\, \text{At}(E_G)\,=\, (\text{ad}(E_G)\oplus \text{At}(E_H))/\text{ad}(E_H)\,
\longrightarrow\, \text{ad}(E_G)\, ,\ \ (v,\,w)\,\longmapsto\, v+\beta(w)
$$
has the property that the composition
$$
\text{ad}(E_G)\, \hookrightarrow\, \text{At}(E_G)\, \stackrel{\beta'}{\longrightarrow}\,
\text{ad}(E_G)
$$
coincides with the identity map of $\text{ad}(E_G)$, where the inclusion of $\text{ad}(E_G)$
in $\text{At}(E_G)$ is the one occurring in the Atiyah exact sequence for $E_G$. Therefore, $\beta'$
produces a holomorphic splitting of the Atiyah exact sequence for $E_G$. Hence $\beta'$ is a
holomorphic connection on $E_G$ \cite{At}.

The Cartan geometry $\beta$ is called \textit{flat} if the curvature of the connection on $E_G$
defined by $\beta'$ vanishes identically.

\subsection{Cartan geometry on a complex torus}\label{se2.2}

We now take $X$ to be a compact complex Lie group, so $X$ is a complex torus. For any $x\,\in\, X$, let
$$
t_x\, :\, X\, \longrightarrow\, X\, , \ \ y\,\longmapsto\, y+x
$$
be the translation by $x$. A holomorphic Cartan geometry $(E_H,\, \beta)$ on $X$ of type
$(G,\, H)$ is called translation invariant if for every $x\,\in\, X$, there is a holomorphic
isomorphism of principal $G$--bundles
$$
\delta_x\, :\, E_G\, \longrightarrow\, t^*_x E_G
$$
such that
\begin{enumerate}
\item $\delta_x(E_H)\,=\, E_H$, and

\item $\delta^*_x\beta\,=\,\beta$.
\end{enumerate}

A conjecture in \cite{DM} says that any flat holomorphic Cartan geometry on a complex torus is
translation invariant (see the first paragraph in the introduction \cite[p.~1]{DM}).

A complex Lie group $G$ will be called \textit{affine} if there is a holomorphic homomorphism
$$
\alpha\, :\, G\, \longrightarrow\, \text{GL}(N, {\mathbb C})
$$
for some positive integer $N$, such that the corresponding homomorphism of Lie algebras
$$
d\alpha\, :\, \text{Lie}(G) \, \longrightarrow\, \text{Lie}(\text{GL}(N, {\mathbb C}))\,=\,
\text{M}(N, {\mathbb C})
$$
is injective.

We will prove that every holomorphic Cartan geometry of type $(G,\, H)$, with $G$ 
affine, on any complex torus is translation invariant. This would imply that they are 
all constructed in the following way.

\subsection{Invariant Cartan geometry on a complex torus}

Let $\widetilde{X}$ be the universal cover of the complex group $X$. The complex Lie group 
$\widetilde{X}$ acts holomorphically on $X$ via translations. Let $E_H$ be a holomorphic 
principal $H$--bundle $E_H$ on $X$ equipped with a holomorphic lift of the action of 
$\widetilde{X}$ on $X$ such that the actions of $H$ and $\widetilde{X}$ on $E_H$ commute. This 
action of $\widetilde{X}$ on $E_H$ produces a flat holomorphic connection on the principal 
$H$--bundle $E_H$. This flat connection on $E_H$ will be denoted by $\nabla^{H}$.

As before, $E_G\,:=\, E_H\times^H G \,\longrightarrow\, X$ is the holomorphic
principal $G$--bundle obtained by extending the structure group of $E_G$. The
holomorphic connection on $E_G$ induced by the above connection $\nabla^{H}$ will
be denoted by $\nabla^{G}$. Let $V_0\, :=\, T_0 X$ be the Lie algebra of $X$. 
Take any holomorphic section
$$
\theta\, \in\, H^0(X,\, \text{ad}(E_G)\otimes V^*_0)\, .
$$
Note that $\theta$ is a holomorphic $1$-form on $X$ with values in $\text{ad}(E_G)$.
Therefore, $\nabla^{G}+\theta$ is a holomorphic connection on $E_G$.

Assume that $\theta$ is flat with respect to the connection on
$\text{ad}(E_G)\otimes V^*_0$ induced by the connection $\nabla^{G}$ on $E_G$
together with the
trivial connection on the trivial vector bundle $X\times V^*_0$.

Note that $\nabla^{H}$ defines a holomorphic $1$-form on the total space of $E_H$ 
with values in $\mathfrak h$. On the other hand, $\theta$ defines a holomorphic 
$1$-form on $E_H$ with values in $\mathfrak g$. Therefore, $\nabla^{H}+\theta$ is 
a holomorphic $1$-form on $E_H$ with values in $\mathfrak g$.
Assume that the $\text{ad}(E_G)$-valued $1$-form $\theta$ satisfies the condition that this
form $\nabla^{H}+\theta\, :\, TE_H\, \longrightarrow\, E_H\times \mathfrak g$ is
an isomorphism.

It is evident that the pair $(E_H,\, \nabla^{H}+\theta)$ defines a Cartan 
geometry on $X$ of type $(G,\, H)$. It is straight-forward to check that this Cartan geometry is 
translation invariant.

The result mentioned in Section \ref{se2.2} that 
every holomorphic Cartan geometry of type $(G,\, H)$, with $G$
affine, on $X$ is translation invariant, in fact implies that if $G$ is
affine, then all holomorphic Cartan geometries of type $(G,\, H)$ on $X$ are
of the type described above. This would be elaborated in Section \ref{se3.3}.

\section{Principal bundles with holomorphic connection over a torus}

\subsection{A canonical flat connection}

As before, $X$ is a compact complex torus. Let $G$ be a connected complex affine group and $E_G$ a
holomorphic principal $G$--bundle over $X$. In \cite{BG} the following was proved:

If $E_G$ admits a holomorphic connection, then it admits a flat holomorphic connection; see 
\cite[p.~41, Theorem 4.1]{BG}.

It should be clarified that in \cite[Theorem 4.1]{BG} it is assumed 
that $G$ admits a holomorphic embedding into some linear group $\text{GL}(N, {\mathbb C})$. 
However, if $p\, :\, G\, \longrightarrow\, G'$ is a holomorphic homomorphism of complex Lie groups 
that produces an isomorphism of Lie algebras, then the holomorphic connections on a holomorphic 
principal $G$--bundle $E_G$ are in bijection with the holomorphic connections on the associated 
holomorphic principal $G'$--bundle $E_{G'}\, :=\, E_G\times^G G'$. Indeed, this follows immediately 
from that fact that the Atiyah bundle and the Atiyah exact sequence for $E_G$ are canonically 
identified with the Atiyah bundle and the Atiyah exact sequence respectively for $E_{G'}$. Therefore,
the above mentioned result of \cite[Theorem 4.1]{BG} for $E_{G'}$ implies that it also holds
for $E_G$.

Assume now that $E_G$ admits a holomorphic connection. In \cite[p.~41, Theorem 4.1]{BG} it was
proved that $E_G$ admits a canonical flat connection. Indeed, $E_G$ is pseudostable and its
characteristic classes of degree one and two vanish (see the fourth statement in
\cite[p.~41, Theorem 4.1]{BG}). Now setting the zero Higgs field on $E_G$, from \cite[p.~20,
Theorem 1.1]{BG} we conclude that $E_G$ has a canonical flat connection; this
canonical connection on $E_G$ will be denoted by $\nabla^{E_G}$. This connection $\nabla^{E_G}$
enjoys the following properties:

Let $G\, \longrightarrow\, M$ be a holomorphic homomorphism of affine groups, and let $$E_M\,:=\,
E_G\times^G M\, \longrightarrow \, X$$ be the associated holomorphic principal $M$--bundle. Then the
canonical connection $\nabla^{E_M}$ on $E_M$ coincides with the one induced by $\nabla^{E_G}$.
Now take $M\,=\, \text{GL}(n, {\mathbb C})$; the holomorphic connection, induced by
$\nabla^{E_M}$, on the rank $n$ vector bundle $E_n$ associated to $E_M$ by the standard
representation of $\text{GL}(n, {\mathbb C})$ will be denoted by $\nabla^{E_n}$. If $V$ is a
pseudostable vector bundle on $X$ with $c_1(V)\,=\, 0\, =\, c_2(V)$, and
$$\phi\, :\, V\, \longrightarrow\, E_n$$ is any holomorphic homomorphism of vector bundles, then
$\phi$ is flat with respect to $\nabla^{E_n}$ and the canonical connection on $V$ (see
\cite{Si}, \cite{BG}.) Below we briefly recall from \cite{Si} and \cite{BG}.

A pseudostable vector bundle $W$ with vanishing Chern classes has a canonical flat connection which is
obtained by setting the zero Higgs field on $W$ \cite[p.~36, Lemma~3.5]{Si}. This canonical connection
on such vector bundles is compatible with the operations of direct sum, tensor product, dualization,
coherent sheaf homomorphisms etc. From these properties it can be deduced that any pseudostable
principal bundle with vanishing characteristic classes has a canonical flat connection
\cite[p.~20, Theorem 1.1]{BG}.

\subsection{Flat connection and translation invariance}

Let $p\, :\, \widetilde{X}\, \longrightarrow\, X$ be the universal cover, so $\widetilde X$
is a complex Lie group isomorphic to ${\mathbb C}^d$, where $d\,=\, \dim_{\mathbb C} X$. Let
$$
\varphi\, :\, \widetilde{X}\times X \, \longrightarrow\, X
$$
be the holomorphic map defined by $(y,\, x) \, \longmapsto\, x+p(y)$. Define the map
$$
\widetilde{\varphi}\, :\, {\mathbb R}\times \widetilde{X}\times X \, \longrightarrow\, X\, ,\ \
(\lambda,\, y,\, x) \, \longmapsto\, x+p(\lambda\cdot y)\, .
$$
Let $E_G$ be a holomorphic principal $G$--bundle on $X$ equipped with a flat connection $\nabla^G$. 
Consider the flat principal $G$--bundle $(\widetilde{\varphi}^*E_G,\, 
\widetilde{\varphi}^*\nabla^G)$ on ${\mathbb R}\times \widetilde{X}\times X$. The flat principal 
$G$--bundles on $\widetilde{X}\times X$ given by $(\widetilde{\varphi}^*E_G,\, 
\widetilde{\varphi}^*\nabla^G)\vert_{\{0\}\times \widetilde{X}\times X}$ and 
$(\widetilde{\varphi}^*E_G,\, \widetilde{\varphi}^*\nabla^G)\vert_{\{1\}\times \widetilde{X}\times 
X}$ are canonically identified by taking parallel translations along the paths $\lambda\, 
\longmapsto\, (\lambda,\, y,\, x)$, $\lambda\, \in\, [0,\, 1]$, in ${\mathbb R}\times \widetilde{X}\times X$. Note that
$(\widetilde{\varphi}^*E_G,\, \widetilde{\varphi}^*\nabla^G)\vert_{\{1\}\times \widetilde{X}\times X}$
is identified with the pullback $({\varphi}^*E_G,\, {\varphi}^*\nabla^G)$, while
$(\widetilde{\varphi}^*E_G,\, \widetilde{\varphi}^*\nabla^G)\vert_{\{0\}\times \widetilde{X}\times X}$
is identified with the pullback $(p^*_XE_G,\, p^*_X\nabla^G)$, where $p_X$ is the natural
projection of $\widetilde{X}\times X$ to $X$.

It is straightforward to check that the above isomorphism between $({\varphi}^*E_G,\, 
{\varphi}^*\nabla^G)$ and $(p^*_XE_G,\, p^*_X\nabla^G)$ produces a translation invariance structure 
on $E_G$ that preserves the connection $\nabla^G$.

\subsection{Translation invariance of Cartan geometries}\label{se3.3}

Let $G$ be a complex affine Lie group. Let $(E_H, \beta)$ be a holomorphic Cartan
geometry of type $(G,\, H)$ on a compact complex torus $X$. Let $\nabla^G$ be the holomorphic
connection on the principal $G$--bundle $E_G\,=\, E_H\times^H G$ defined by $\beta$. Let
$\nabla^{G,0}$ be the canonical flat connection on $E_G$. So we have
\begin{equation}\label{e2}
\theta\,:=\,\nabla^G- \nabla^{G,0}\,=\, H^0(X,\, \text{ad}(E_G)\otimes\Omega^1_X)\, ,
\end{equation}
where $\text{ad}(E_G)\,=\, E_G\times^G {\mathfrak g}$ is the adjoint bundle for $E_G$.

\begin{lemma}\label{lem1}
The translation invariance structure on $E_G$ given by the flat connection $\nabla^{G,0}$
preserves the holomorphic connection $\nabla^G$.
\end{lemma}

\begin{proof}
Let ${\nabla}^{\rm ad}$ be the flat holomorphic connection on
the adjoint vector bundle $\text{ad}(E_G)$ induced by the 
flat holomorphic connection $\nabla^{G,0}$ on $E_G$. Let $\widetilde{\nabla}^{\rm ad}$ be the 
connection on $\text{ad}(E_G)\otimes\Omega^1_X$ given by the above connection 
${\nabla}^{\rm ad}$ on $\text{ad}(E_G)$ and the unique trivial connection on 
$\Omega^1_X$. To prove the lemma it suffices to show that the section $\theta$ in \eqref{e2} is flat 
with respect to this connection $\widetilde{\nabla}^{\rm ad}$ on $\text{ad}(E_G)\otimes\Omega^1_X$. 

To prove that $\theta$ is flat, first note that $\widetilde{\nabla}^{\rm ad}$ is the canonical flat 
connection on $\text{ad}(E_G)\otimes\Omega^1_X$, because
${\nabla}^{\rm ad}$ is the canonical flat connection on $\text{ad}(E_G)$ and the trivial 
connection on $\Omega^1_X$ is the canonical flat connection on $\Omega^1_X$. Therefore, any holomorphic 
section of $\text{ad}(E_G)\otimes\Omega^1_X$ is flat with respect to the connection 
$\widetilde{\nabla}^{\rm ad}$. In particular, the section $\theta$ in \eqref{e2} is flat with
respect to $\widetilde{\nabla}^{\rm ad}$.
\end{proof}

\begin{lemma}\label{lem2}
The translation invariance structure on $E_G$ given by the flat connection $\nabla^{G,0}$
preserves the reduction $E_H\, \subset\, E_G$.
\end{lemma}

\begin{proof}
We know that $E_G$ is pseudostable and its characteristic classes vanish \cite{BG}. Let 
$$\text{ad}(E_H)\, \subset\, \text{ad}(E_G)$$ be the adjoint bundle for $E_H$. From \eqref{e1} 
we know that the quotient bundle $\text{ad}(E_G)/\text{ad}(E_H)$ is isomorphic to the 
holomorphic tangent bundle $TX$ of the torus and, consequently, it is trivial. In particular, 
$\text{ad}(E_G)/\text{ad}(E_H)$ is pseudostable and its Chern classes vanish. Therefore, the 
sub-vector bundle $\text{ad}(E_H)$ is also pseudostable and its Chern classes vanish. From these 
it follows that $E_H$ is pseudostable and its characteristic classes vanish.

Let $\nabla^{H,0}$ be the canonical flat connection on $E_H$. The canonical connection
$\nabla^{G,0}$ on $E_G$ is induced by $\nabla^{H,0}$. This implies that the reduction $E_H\, \subset\,
E_G$ is preserved by $\nabla^{G,0}$.
\end{proof}

%%%%%%%%%%%%%%%%%%%%%%%%%%%%%%%%%%%%%%%%%%%%%%%%%%%%%%%%%%%%%%%%%

\end{document}